\documentclass[10pt]{amsart}
\usepackage{qtree,array,youngtab}
\usepackage{color, graphicx, comment}

\newcommand{\DASH}{{\textbf{\Large $\mid$}}}

\definecolor{darkgreen}{rgb}{0.,0.5,0.}
\newcommand{\bO}{\textcolor{blue}{0}}
\newcommand{\rO}{\textcolor{red}{0}}
\newcommand{\gO}{\textcolor{darkgreen}{0}}

\newcommand{\bI}{\textcolor{blue}{1}}
\newcommand{\rI}{\textcolor{red}{1}}
\newcommand{\gI}{\textcolor{darkgreen}{1}}
\DeclareMathOperator{\modsymb}{mod}
\newcommand{\tcore}{{\text{$t$-core}}}
\allowdisplaybreaks

\numberwithin{equation}{section} \overfullrule 5pt
\newtheorem{thm}{Theorem}[section]
\newtheorem{cor}[thm]{Corollary}
\newtheorem{lem}[thm]{Lemma}

\theoremstyle{definition}
\newtheorem{defi}{Definition}[section]

\title[Difference operators under the Littlewood decomposition]{Difference operators for partitions under the Littlewood decomposition}

\author{Paul-Olivier Dehaye}   
\address[Paul-Olivier Dehaye]{I-Math, Universit\"at Z\"urich,
  Winterthurerstrasse 190, Z\"urich 8057,
Switzerland} 
\email{paul-olivier.dehaye@math.uzh.ch}

\author{Guo-Niu Han}
\address[Guo-Niu Han]{I.R.M.A., UMR 7501, Universit\'e de Strasbourg
et CNRS, 7 rue Ren\'e Descartes, F-67084 Strasbourg, France}
\email{guoniu.han@unistra.fr}

\author[Huan Xiong]{Huan Xiong$^*$}
\thanks{$^*$ Huan Xiong is the corresponding author.}
\address[Huan Xiong]{I.R.M.A., UMR 7501, Universit\'e de Strasbourg
et CNRS, 7 rue Ren\'e Descartes, F-67084 Strasbourg, France}
\email[Corresponding author]{xiong@math.unistra.fr}

\subjclass[2010]{05A15, 05A17, 05A19,  05E05, 05E10, 11P81}

\keywords{Littlewood decomposition, partition, Plancherel measure,
difference operator}

\begin{document}
\begin{abstract} 
The concept of $t$-difference operator
for functions of partitions is introduced to prove a generalization of
Stanley's theorem on polynomiality of Plancherel averages of symmetric functions related to contents
and hook lengths. 
Our extension uses a generalization of the
notion of Plancherel measure, based on walks in the Young lattice
with steps given by the addition of $t$-hooks. It is well-known that
the hook lengths of multiples of $t$ can be characterized by the Littlewood decomposition. Our study gives some further information on the contents and hook lengths of other 
congruence classes modulo $t$.
\end{abstract}
\maketitle

\section{Introduction} \label{sec:introduction}
\subsection{Goal}
The purpose of this paper is to compute explicit averages over partitions of some statistics of their invariants, such as contents and hook lengths, and establish polynomiality of these averages in the size of the partitions considered.
In contrast with previous results, we restrict our attention to subsets of partitions classified by the divisibility properties of those invariants, and we modify the averaging measures accordingly. 
Our main result is stated in Theorem~\ref{th:main'}.

\subsection{Basic definitions}
A \emph{partition} is a finite weakly decreasing sequence of positive integers $\lambda = $
$(\lambda_1$, $\lambda_2, \cdots$, $ \lambda_r)$. The \emph{size}
of the partition $\lambda$ is defined by the  integer $| \lambda |=\sum_{1\leq i\leq r}\lambda_i$. A partition $\lambda$ could be
represented by its \emph{Young diagram}, which is a collection of boxes
arranged in left-justified rows with $\lambda_i$ boxes in the $i$-th
row. The \emph{content} of a box $\square=(i,j)$ in a partition $\lambda$ is
defined by $c_{\square}=j-i$. 
With each $(i, j)$-box is associated its \emph{hook length},
denoted by $h_{(i, j)}$, which is the number of boxes exactly to the
right, or exactly below, including the box itself.  
The hook length multi-set of $\lambda$ is denoted by $\mathcal{H}(\lambda)$. 
In Figure~\ref{YoungFig}, 
the Young diagram and hook lengths of the partition (6, 3, 2, 2) are
represented. For example, $h_{(1,2)}=8$ and~$h_{(3,1)}=3$.

\begin{figure}[htbp]
\begin{center}
\Yvcentermath1

\begin{tabular}{c}
$\young(985321,541,32,21)$

\end{tabular}

\end{center}
\caption{The Young diagram of the partition $(6,3,2,2)$ and the hook
lengths of corresponding boxes. }
\label{YoungFig}
\end{figure}

Let $t$ be a positive integer. A partition is called a
\emph{$t$-core partition} if none of its hook lengths is divisible
by $t$. For example, from Figure \textbf{$1$} we can see  that $\lambda=(6,3,2,2)$ is a
$7$-core partition.
Let 
$$
{\mathcal{H}}_t(\lambda)=\{ h\in\mathcal{H}(\lambda): \ h\equiv 0 \pmod t\}
$$
be  the multi-set
of  hook lengths of multiples of $t$.

\subsection{01-sequences, Littlewood decomposition and Plancherel averages}
Each partition $\lambda$ can be associated with its \emph{01-sequence} $z(\lambda)=(z_{\lambda,i})_{i\in \mathbb{Z}}=(z_{i})_{i\in \mathbb{Z}}$ as follows. 
First, consider the edges on the  boundary  of $\lambda$,
starting at the bottom and ending to the right. 
Label the vertical
(resp.~horizontal) edges  0 (resp.~1).  
We
obtain a bi-infinite 01-sequence $(z_i)_{i\in \mathbb{Z}}$ 
 beginning with infinitely many 0's
 and ending with infinitely many 1's.  
 The index $0$ of the bi-infinite 01-sequence is defined by 
 the condition $\#\{z_i =0: i
\geq 0\}= \#\{z_i =1: i< 0\} $.  
Every partition and  its  01-sequence are uniquely determined
by each other (see Figure~\ref{fig.sequence} for a running example of the constructions presented in this section).

\begin{figure}[htbp]
\begin{center}
\includegraphics[width=0.78\textwidth]{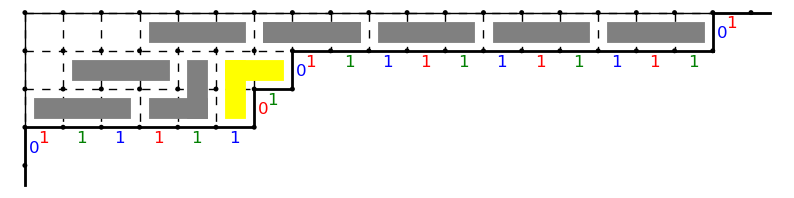} 
\caption{This is our running example for the 01-sequence and the Littlewood decomposition of a partition. The Young diagram of the partition $\lambda = (18,7,6)$ is drawn, together with its 01-sequence $\cdots\bO                 \rI\gI\bI \DASH \rI\gI\bI \rO\gI\bO  \rI\gI\bI \rI\gI\bI \rI\gI\bI \rI\gI\bO \rI\cdots$ (the $\DASH$ indicates the unique location where the number of 1's to the left equals the number of 0's to the right, which is between index -1 and 0 in the 01-sequence; visually this corresponds to the main diagonal in the Young diagram). Notice the mnemonic ``RGB'' for the colourings, starting after the $\DASH$ mark. The grey shapes indicate a sequence of 3-hooks that can be removed from $\lambda$, to finally obtain the 3-core $(3,1)$ of $\lambda$. One could start with the yellow shape for instance, and then make various choices for the order of the removals. The process always ends with the same $3$-core $(3,1)$, irrespective of the order of removals.
       Quotients will be $\cdots\rO\rI\DASH\rI\rO\rI\rI\rI\rI\rI\cdots$ (partition $(2)$), $\cdots\gO\DASH\gI\gI\gI\gI\gI\gI\gI\cdots$ (empty partition $\emptyset$) and $\cdots\bO\bI\bI\DASH\bO\bI\bI\bI\bO\bI\cdots$ (partition $(5,2)$).     
        \label{fig.sequence}}
    \label{fig.01}
  \end{center}
\end{figure}

The 01-sequences contain a lot of information about  partitions (see \cite{arm, James}).  By the construction of  01-sequences the following lemma is obvious.
\begin{lem}\label{th:01sequencehookcontent}
For a partition $\lambda$ and its 01-sequence
$z(\lambda)=(z_i)_{i\in \mathbb{Z}}$, we have
$$
|\lambda|=\#\{(i,j):i<j,\ z_i=1,\ z_j=0\}.
$$
Actually, each box in $\lambda$ is uniquely determined by such a
pair $(i,j).$\footnote{\label{foot:box} Notice that here $(i,j)$ is not the coordinate of the box in $\lambda$.} The hook length of such box is $j-i$.
Also, for an edge labeled by $z_i=0,$  the box to the left has content
$i$; for an edge labeled by $z_i=1,$  the box above has content
$i+1$.
\end{lem}

By Lemma \ref{th:01sequencehookcontent} a partition
$\lambda$ is a $t$-core partition if and only if there are no $i\in
\mathbb{Z}$ satisfying $z_{i} = 1$, $z_{i+t} = 0$ in the 01-sequence
$z(\lambda)$. If $z_{i} = 1$, $z_{i+t} = 0$ (resp. $z_{i} =
0$, $z_{i+t} = 1$) in a 01-sequence $z(\lambda)=(z_{j})_{j\in
\mathbb{Z}}$, 
by exchanging the values of $z_i$ and $z_{i+t}$ we obtain a new
01-sequence and thus a new partition. We say that this new partition
is obtained by removing (resp.~adding) a {\it $t$-hook}  from (resp.~to)
$\lambda$.

\medskip

Starting with a partition
$\lambda$, we can always remove several $t$-hooks  until no more
$t$-hooks could be removed. 
The final
partition we get by this operation is called the
{\it $t$-{core}} of $\lambda$ and denoted by $\lambda_{\tcore}$,
which is actually independent of the
ordering of removing $t$-hooks from $\lambda$ (see \cite{arm, James}).
For $0\le i \leq t-1$, the $i$-th {\it $t$-quotient}  of the partition
$\lambda$ is defined to be the partition associated with the
01-sequence obtained by taking the subsequence  $(z_{t j
+i})_{j\in \mathbb{Z}} $ of $z(\lambda)$, which is denoted by
$\lambda^i.$ 
For each positive integer $t$, we actually build a map
$$
\{\text{partitions} \} \rightarrow \{ t\text{-core partitions} \} \times \{
\text{partitions} \}^t 
$$ 
which sends a partition $\lambda$ to
$(\lambda_\tcore; \lambda^0,\lambda^1,\ldots,\lambda^{t-1})$.
This
map is a bijection,
 called the \emph{Littlewood decomposition of partitions at $t$}
\cite{Macdonald}. 
In this paper, we always set that a partition $\lambda$ is identical with its
image under the Littlewood decomposition, i.e., we always write
$\lambda=(\lambda_\tcore;
\lambda^0,\lambda^1,\ldots,\lambda^{t-1})$.
The Littlewood decomposition has the following two fundamental properties (see \cite{Macdonald}): 
$$|\lambda|=|\lambda_{\tcore}|+t(|\lambda^0|+|\lambda^1|+\cdots+|\lambda^{t-1}|)$$
and
$$\{h/t\mid h\in \mathcal{H}_t(\lambda)\}=
\mathcal{H}(\lambda^0)\cup \mathcal{H}(\lambda^1)\cup \cdots \cup \mathcal{H}(\lambda^{t-1}).$$

\medskip

Let $f_\lambda$ (resp. $f_{\lambda/\mu}$) be the number of standard
Young tableaux of shape $\lambda$ (resp. $\lambda/\mu$) and
$H_{\lambda}=\prod_{\square\in\lambda} h_{\square}$ be the product
of all hook lengths of boxes in $\lambda$. For convenience we set
$f_\emptyset=1$ and $H_\emptyset=1$ for the empty partition
$\emptyset$. It is well known that (see \cite{frt, han2, rsk, ec2})
\begin{equation}
\label{eq:hookformula} f_\lambda = \frac{|\lambda| !}{H_{\lambda}}
\qquad \text{and}\qquad \frac{1}{n!}\sum_{|\lambda| =n}
f_\lambda^2=1.
\end{equation}
The latter identity defines a measure on partitions of size $n$, called
the Plancherel measure. Alternatively, we can see $f_\lambda$ as counting ``walks" among Young diagrams up to the shape $\lambda$ (so-called \emph{hook walks}, see \cite{walk1, walk2}).  

\subsection{Polynomiality of Plancherel averages and modified hook walks}
Expanding on a formula of Nekrasov and Okounkov on hook lengths \cite{no}, the
second author conjectured \cite{han} that
$$
P(n)=\frac{1}{n!}\sum_{|\lambda|= n}
f_\lambda^2\sum_{\square\in\lambda}h_{\square}^{2k} 
$$ is always a
polynomial of $n$ for any $k\in \mathbb{N}$. This conjecture  was generalized and
proved by  Stanley \cite{stan} (see also \cite{tew, clps, han3, panova}), and later generalized in~\cite{hanxiong, hanxiong1, hanxiong4}.
Let $Q$ be a symmetric function in infinitely many variables and $E$ be a finite set with $n$ elements.
The symbol $Q(x: x\in E)$
means that $n$ of
the variables are substituted by $x$ 
for $x\in E$, and all other variables by~$0$.

\begin{thm}[Stanley \cite{stan}]\label{th:stanley}
Let $Q$ be a symmetric function in infinitely many variables. Then 
\begin{equation}\label{eq:stanleycontent}
P(n)=\frac{1}{n!}\sum_{|\lambda|= n} f_\lambda^2\, Q(c_{\square}:
{\square}\in\lambda)
\end{equation}
and
\begin{equation}\label{eq:stanleyhook}
P^*(n)=\frac{1}{n!}\sum_{|\lambda|= n} f_\lambda^2\, Q(h_{\square}^2:
{\square}\in\lambda)
\end{equation}
are polynomials in $n$. 
\end{thm}

Olshanski \cite{ols2} also proved the content case of Theorem
\ref{th:stanley}.

\medskip

Let $\lambda$ and $\mu$ be two partitions. We say that
$\lambda\geq_t \mu$ if $\lambda$ could be obtained by adding some
$t$-hooks to $\mu$. Then, the set of all partitions becomes a
partially ordered set under this relation.
If $\lambda\geq_t \mu$, define $F_{\lambda/\mu}$ to be the number of ways of removing $t$-hooks:
\begin{equation}
F_{\mu/\mu}:=1
\qquad \text{and}\qquad F_{\lambda/\mu}:=\sum\limits_{\substack{\lambda\geq_t\lambda^-\geq_t\mu\\
|\lambda/\lambda^-|=t }}F_{\lambda^-/\mu}\qquad \text{(for $\lambda \neq \mu$)}.
\end{equation}
It is easy to see that
$F_{\lambda/\mu}=\#\{(P_0,P_1,\ldots,P_{t-1}):$ $P_i$  \emph{is a 
Young tableau of shape}  $\lambda^i/\mu^i$, \emph{the union of
entries in} $P_0,P_1,\ldots,P_{t-1}$ \emph{are} $1,2,\ldots,
\sum_{i=0}^{t-1}|\lambda^i/\mu^i|\}$.
Hence,
$$F_{\lambda/\mu}=\binom{\sum_{i=0}^{t-1}|\lambda^i/\mu^i|}
{|\lambda^0/\mu^0|,|\lambda^1/\mu^1|,\ldots,|\lambda^{t-1}/\mu^{t-1}|}\prod_{i=0}^{t-1}
 f_{\lambda^i/\mu^i}.$$
These $F_{\lambda/\mu}$ can be seen as counting hook walks distributed among the quotient partitions of $\lambda$.  
Let \begin{equation}
F_\lambda:=
F_{\lambda/ \lambda_{\tcore}}=
\binom{n}
{|\lambda^0|,|\lambda^1|,\cdots,|\lambda^{t-1}|}\prod_{i=0}^{t-1}
f_{\lambda^i} = {n! t^n}{G_{\lambda}},
 \end{equation}
 where $n=|\lambda^0|+|\lambda^1|+\ldots+|\lambda^{t-1}|$
 and 
 $$G_{\lambda}:=\frac {1}{\prod_{h \in \mathcal{H}_t(\lambda)} h}.$$

When $t=1$,  we have $F_\lambda=f_\lambda$ and $G_\lambda = 1/H_\lambda$.
Also, when $\lambda$ is a $t$-core partition, we have $F_\lambda=G_\lambda=1$.

\subsection{Difference operators and main result}

Let $g$ be a function of partitions and $\lambda$ a partition.
The difference operator~$D$ for partitions was defined in
\cite{hanxiong} as
\begin{equation*} 
    Dg(\lambda)=\sum_{|\lambda^{+}/\lambda|=1}g(\lambda^{+})-g(\lambda).
\end{equation*}
In this paper, we introduce a generalization of $D$. For every integer
 $t$, let
\begin{equation}\label{eq:diffDt}
    D_tg(\lambda)=\sum\limits_{\substack{\lambda^{+}\geq_t\lambda\\ |\lambda^{+}/\lambda|=t}}g(\lambda^{+})-g(\lambda).
\end{equation}
Accordingly, $D_1=D$.
 
The higher-order difference operators for $D_t$ are defined by induction:
$$D_t^0g:=g \text{\quad and\quad } D_t^k g:=D_t(D_t^{k-1} g)\quad (k\geq 1).$$ 
These operators also fit in Stanley's theory of differential posets \cite{stan1}, although this is a language we will not use here.
The following theorem is our main result, which will be proved in
Section \ref{sec:ProofMain}.

\begin{thm}\label{th:main'}
Suppose that $t$ is a positive integer, $u',v', j_u,{j'}_v, k_u,{k'}_v$ are nonnegative integers and $\mu$ is a given $t$-core
partition. Then, there exists some fixed $r\in \mathbb{N}$ such that
$$
D_t^r\Biggl(G_\lambda \Biggl(\prod_{u=1}^{u'}\sum\limits_{\substack{\square\in \lambda\\
h_{\square}\equiv \pm j_u (\modsymb t
)}}h_{\square}^{2k_u}\Biggr)\Biggl(\prod_{v=1}^{v'}\sum\limits_{\substack{\square\in \lambda\\
c_{\square}\equiv j'_v (\modsymb t )}}c_{\square}^{k'_v}\Biggr) \Biggr)=0
$$
for every partition $\lambda$ with $\lambda_\tcore=\mu.$
Moreover,
\begin{equation*}
P(n)=\sum\limits_{\substack{\lambda_\tcore=\mu\\
|\lambda/\mu|=nt }} F_{\lambda/\mu} G_\lambda\,
\Biggl(\prod_{u=1}^{u'}\sum\limits_{\substack{\square\in \lambda\\
h_{\square}\equiv \pm j_u  (\modsymb t
)}}h_{\square}^{2k_u}\Biggr)\Biggl(\prod_{v=1}^{v'}\sum\limits_{\substack{\square\in \lambda\\
c_{\square}\equiv j'_v (\modsymb t )}}c_{\square}^{k'_v}\Biggr)
\end{equation*}
is a polynomial of  $n$ ($n$ is a nonnegative integer).

\end{thm}

Let $Q_1$ and $Q_2$ be two symmetric functions.  By Theorem \ref{th:main'}, 
there exists some fixed $r\in \mathbb{N}$ such that
$$
D_t^r\bigl(G_\lambda Q_1(h_{\square}^2:
{\square}\in\lambda)Q_2(c_{\square}: {\square}\in\lambda) \bigr)=0
$$
for every partition $\lambda$ with $\lambda_\tcore=\mu.$
Moreover,
\begin{equation} \label{eq:main'}
P(n)=\sum\limits_{\substack{\lambda_\tcore=\mu\\
|\lambda/\mu|=nt }} F_{\lambda/\mu} G_\lambda\, Q_1(h_{\square}^2:
{\square}\in\lambda)Q_2(c_{\square}: {\square}\in\lambda)
\end{equation}
is a polynomial of  $n$.
When $t=1,\mu=\emptyset$, we have \eqref{eq:main'} implies
Theorem \ref{th:stanley}.

\subsection{Specialization for the square case}
We now focus on expressions of square sum of contents or hook lengths, 
and thereby obtain
explicit results. 
Details of 
the proofs are given in Section \ref{sec:hanconjecture}. 
We start by considering individual partitions. 
\begin{thm}\label{th:hookcontentsquare:mu0}
Let $0\leq k\leq t-1.$ For every partition $\lambda$ with $\lambda_\tcore=\emptyset$ we
 have
 \begin{align}\label{eq:h2c2}\ & \sum\limits_{\substack{\square\in\lambda\\
h_\square\equiv
k(\modsymb t )}}h_\square^2+\sum\limits_{\substack{\square\in\lambda\\
h_\square\equiv t-k(\modsymb t )}}h_\square^2 -
\Biggl(\sum\limits_{\substack{\square\in\lambda\\ c_\square\equiv
k(\modsymb t )}}c_\square^2+\sum\limits_{\substack{\square\in\lambda\\
c_\square\equiv t-k(\modsymb t )}}c_\square^2\Biggr)\\&= 2t^2
		\Bigl(	\sum_{i=0}^{t-1-k} n_i n_{i+k}
		+\sum_{i=0}^{k-1} n_i n_{i+t-k} \Bigr),\nonumber
\end{align} 
where $n_i=|\lambda^i|$.
\end{thm}

By letting $t=1$
we obtain the following well-known result (see \cite{Macdonald}).

\begin{cor}\label{th:hookcontentsquare2}
 For every partition $\lambda$ we have
\begin{align*}
\sum\limits_{\substack{\square\in\lambda}}h_\square^2 - \sum\limits_{\substack{\square\in\lambda\\
}}c_\square^2 = |\lambda|^2.
\end{align*}
\end{cor}

\medskip

We now turn to $t$-Plancherel averages, introducing $F_{\lambda/ \mu}$ and $G_\lambda$ as weights in the previous results and summing over $\lambda$.
The results are stated in two cases according to the divisibility of the hook lengths by $t$.

\begin{thm} \label{th:hooksquare1} Suppose that $\mu$ is a given $t$-core partition and $1\leq k\leq t-1$.
  Then we have
\begin{align*}\ & \sum\limits_{\substack{\lambda_\tcore=\mu\\
|\lambda/\mu|=nt }}F_{\lambda/\mu}G_\lambda\Biggl(\sum\limits
_{\substack{\square\in\lambda\\
h_\square\equiv
k(\modsymb t )}}h_\square^2+\sum\limits_{\substack{\square\in\lambda\\
h_\square\equiv t-k(\modsymb t
)}}h_\square^2\Biggr)\\
&=6t\binom{n}{2}+\Bigl(2k(t-k)
+4t|\mu(k)|+4t|\mu(t-k)|\Bigr)n\\
&\qquad\qquad +\Biggl(\sum\limits_
{\substack{\square\in\mu\\
h_\square\equiv
k(\modsymb t )}}h_\square^2+\sum\limits_{\substack{\square\in\mu\\
h_\square\equiv t-k(\modsymb t )}}h_\square^2\Biggr),
\end{align*} 
where $\mu(k)=\{h\in \mathcal{H}(\mu) \mid h\equiv k \pmod t\}$, viewed as a multiset.
In
particular, let $\mu=\emptyset$. Then we have
\begin{align*}\ & \sum\limits_{\substack{\lambda_\tcore=\emptyset\\
|\lambda|=nt }}F_{\lambda}G_\lambda\Biggl(\sum\limits
_{\substack{\square\in\lambda\\
h_\square\equiv
k(\modsymb t )}}h_\square^2+\sum\limits_{\substack{\square\in\lambda\\
h_\square\equiv t-k(\modsymb t
)}}h_\square^2\Biggr)\\&=6t\binom{n}{2}+2k(t-k)n.\end{align*}
\end{thm}


\begin{thm} \label{th:hooksquare2} Suppose that $\mu$ is a given $t$-core partition.
Then we have \begin{align*} \sum\limits_{\substack{\lambda_\tcore=\mu\\
|\lambda/\mu|=nt }}F_{\lambda/\mu}G_\lambda \sum\limits_{\substack{
\square\in\lambda\\
h_\square\equiv 0(\modsymb t
)}}h_\square^2=nt^2+3t\binom{n}{2}.
\end{align*} 
\end{thm}

By Theorems \ref{th:hooksquare1} and  \ref{th:hooksquare2} we obtain the following result.

\begin{cor} \label{th:hooksquare3} Suppose that $\mu$ is a given $t$-core partition.
  Then \begin{align*}\ & \sum\limits_{\substack{\lambda_\tcore=\mu\\
|\lambda/\mu|=nt }}F_{\lambda/\mu}G_\lambda\sum\limits_{\substack{\square\in\lambda\\
}}h_\square^2=\frac{3t^2n^2}{2}+\frac{nt(t^2-3t-1+24|\mu|)}{6}+\sum\limits_{\substack{\square\in\mu\\
}}h_\square^2.\end{align*}  In
particular, let $\mu=\emptyset$. Then we have
\begin{align*} \sum\limits_{\substack{\lambda_\tcore=\emptyset\\
|\lambda|=nt }}F_{\lambda}G_\lambda
\sum\limits_{\substack{\square\in\lambda}}h_\square^2=
\frac{3t^2n^2}{2}+\frac{nt(t^2-3t-1)}{6}.\end{align*}
\end{cor}

Similar results for contents can also be obtained.

\begin{thm} \label{th:contentsquare11}
Let $0\leq k\leq t-1$. We have 
 \begin{align*} \sum\limits_{\substack{\lambda_\tcore=\emptyset\\
|\lambda|=nt }}F_{\lambda}G_\lambda
\sum\limits_{\substack{\square\in\lambda\\
c_\square\equiv k(\modsymb t)}}c_\square^2= t\binom{n}{2}+
    k(t-k)n.\end{align*}
\end{thm}

Theorem \ref{th:contentsquare11} implies the following identity.
\begin{align}\label{eq:sum_c2} \sum\limits_{\substack{\lambda_\tcore=\emptyset\\
|\lambda|=nt }}F_{\lambda}G_\lambda
\sum\limits_{\substack{\square\in\lambda}}c_\square^2=
t^2\binom{n}{2}+\frac{(t^3-t)n}{6}.
\end{align}
In fact we will derive a more general result in Theorem \ref{th:contentsquare1} by replacing $\emptyset$ by a general $t$-core partition.

\section{Partitions and $t$-difference operator} \label{sec:Dpoly} 
We now turn our attention to preliminary lemmas on the $t$-difference operators~$D_t^k$, and in particular a theorem outlining how polynomiality of $t$-Plancherel averages can be deduced from the vanishing of $D_t^k$ for large enough $k$.
It is known \cite{hanxiong} that
\begin{equation}\label{eq:DH}
D\Bigl(\frac {1}{H_\lambda}\Bigr)=0.
\end{equation}
We establish a similar result for the $t$-difference operator.

\begin{lem} \label{th:Glambda}
Suppose that $\lambda$ is a partition. Then
$$
D_t (G_\lambda)=0.
$$
In other words,
$$G_{\lambda}=\sum\limits_{\substack{\lambda^+\geq_t\lambda\\
|\lambda^+/\lambda|=t }}G_{\lambda^+}.
$$
\end{lem}
\begin{proof}
Suppose that
$$
\lambda=(\lambda_\tcore;
\lambda^0,\lambda^1,\ldots,\lambda^{i-1},\lambda^i,\lambda^{i+1},\ldots,\lambda^{t-1})
$$
and
$$
\lambda^+=(\lambda_\tcore;
\lambda^0,\lambda^1,\ldots,\lambda^{i-1},(\lambda^{i})^+,\lambda^{i+1},\ldots,\lambda^{t-1})
$$
where $|(\lambda^{i})^+/\lambda^{i}|=1.$  Then, by the definition of
$G_\lambda$ we have
$$
\frac{G_{\lambda^+}}{G_\lambda}=
\frac{H_{\lambda^{i}}}{tH_{(\lambda^{i})^+}}.
$$
By  \eqref{eq:DH} we have
$$
\sum_{|(\lambda^{i})^+/\lambda^{i}|=1}
\frac{G_{(\lambda_\tcore;
\lambda^0,\lambda^1,\ldots,\lambda^{i-1},(\lambda^{i})^+,\lambda^{i+1},\ldots,\lambda^{t-1})}}{G_\lambda}=\frac{1}{t}.
$$
Summing the above equality over $0\leq i\leq t-1,$ we prove our
claim.
\end{proof}
We now prove a commutation relation between the averaging over partitions of a given weight and the $D_t$ operator.

\begin{lem} \label{th:telescope}
Suppose that $\mu$ is a  given partition and $g$ is a function    of
 partitions. For every $n\in\mathbb{N}$, let
$$
P_g(n):=\sum\limits_{\substack{\lambda\geq_t\mu\\
|\lambda/\mu|=nt }}F_{\lambda/\mu}g(\lambda).
$$
Then $$P_g(n+1)-P_g(n)=P_{D_tg}(n).$$

\end{lem}
\begin{proof}
The proof is straightforward and follows from expanding definitions:
\begin{align*}
P_g(n+1)-P_g(n)
&=\sum\limits_{\substack{\nu\geq_t\mu\\
|\nu/\mu|=(n+1)t }}F_{\nu/\mu}g(\nu)-\sum\limits_{\substack{\lambda\geq_t\mu\\
|\lambda/\mu|=nt }}F_{\lambda/\mu}g(\lambda)\\
&= \sum\limits_{\substack{\nu\geq_t\mu\\
|\nu/\mu|=(n+1)t }}\sum\limits_{\substack{\nu\geq_t\nu^-\geq_t\mu\\
|\nu/\nu^-|=t }}F_{\nu^-/\mu}g(\nu)-\sum\limits_{\substack{\lambda\geq_t\mu\\
|\lambda/\mu|=nt }}F_{\lambda/\mu}g(\lambda)\\
&=\sum\limits_{\substack{\lambda\geq_t\mu\\
|\lambda/\mu|=nt }}F_{\lambda/\mu}\bigl(\sum\limits_{\substack{\lambda^{+}\geq_t\lambda\\
|\lambda^{+}/\lambda|=t}}g(\lambda^+)-g(\lambda)\bigr) \\
&= P_{D_tg}(n).\qedhere
\end{align*}
\end{proof}

{\it Example}. Let  $g(\lambda)=G_\lambda$. Then $D_tg(\lambda)=0$
by Lemma \ref{th:Glambda}, which means that $P_{D_tg}(n)=0$.
Consequently, $P_g(n+1)=P_g(n)=\cdots=P_g(0)=G_\mu,$ or
\begin{equation}
\sum\limits_{\substack{\lambda\geq_t\mu\\
|\lambda/\mu|=nt }}F_{\lambda/\mu} G_\lambda= G_\mu.
\end{equation}
Recall that when $\mu$ is a $t$-core partition, we have $G_\mu=1$. This implies \begin{equation}
\sum\limits_{\substack{\lambda_\tcore=\mu\\
|\lambda/\mu|=nt }}F_{\lambda/\mu} G_\lambda= 1.
\end{equation}

In particular, let $\mu=\emptyset$. We obtain a generalization of
the second identity of \eqref{eq:hookformula}:

\begin{equation}
\sum\limits_{\substack{\lambda_\tcore=\emptyset\\
|\lambda|=nt }}F_{\lambda}G_\lambda= 1,
\end{equation}
or equivalently,
\begin{equation}
\sum_{|\lambda^0|+|\lambda^1|+\cdots+|\lambda^{t-1}|=n}\binom{n}
{|\lambda^0|,|\lambda^1|,\cdots,|\lambda^{t-1}|}\prod_{i=0}^{t-1}\frac{f_{\lambda^i}^2}{|\lambda^i|!}=t^n.
\end{equation}
Formulas such as this one can be used to create new measures on
partitions, called \emph{$t$-Plancherel}.

Recall that the higher-order difference operators for $D_t$ are defined by
induction $D_t^0g:=g$ and $D_t^k g:=D_t(D_t^{k-1} g)$ ($k\geq 1$).
We write $D_t g(\nu) := D_t g(\lambda) |_{\lambda=\nu}$ for a fixed
partition $\nu$.

\medskip

\begin{thm}\label{th:main3}
Let $g$ be a function of partitions and $\mu$ be a given partition.
Then, 
\begin{equation}\label{eq:main3}
P_g(n)=\sum\limits_{\substack{\lambda\geq_t\mu\\
|\lambda/\mu|=nt }}F_{\lambda/\mu}g(\lambda)
=\sum_{k=0}^n\binom{n}{k}D_t^kg(\mu)
\end{equation}
and
\begin{equation}\label{eq:main3b}
D_t^ng(\mu)=\sum_{k=0}^n(-1)^{n+k}\binom{n}{k}P_g(k).
\end{equation}
In particular, if there exists some positive integer $r$ such that
$D_t^r g(\lambda)=0$ for every partition $\lambda\geq_t \mu$, then $P_g(n)$ is
a polynomial in $n$.
\end{thm}

\begin{proof}
Identity \eqref{eq:main3} will be proved by induction. The case $n=0$
is obvious. Assume that \eqref{eq:main3} is true for some
nonnegative integer $n$. By Lemma \ref{th:telescope} we obtain
\begin{align*}
P_g(n+1)&= P_g(n)+ P_{D_tg}(n)
\\&=\sum_{k=0}^n\binom{n}{k}D_t^kg(\mu)+\sum_{k=0}^n\binom{n}{k}D_{t}^{k+1}g(\mu)
\\&=\sum_{k=0}^{n+1}\binom{n+1}{k}D_t^kg(\mu).
\end{align*}

Identity \eqref{eq:main3b} follows from the famous M\"obius inversion
formula \cite{Rota1964}.
\end{proof}

\section{Corners of partitions} \label{sec:main}

Now we introduce the concept of corners in a partition.  For a
partition $\lambda$, the \emph{outer corners} (see 
\cite{bandlow, hanxiong}) are the boxes which can be removed 
in such a way that after removal the resulting diagram is still a Young diagram.
The coordinates of outer corners are denoted by
$(\alpha_1,\beta_1),\ldots,(\alpha_{m},\beta_{m})$  such that
$\alpha_1>\alpha_2>\cdots
>\alpha_m$. Let $y_j:=\beta_j-\alpha_j$ ($1\leq j \leq m$) be the contents of outer corners.  Set $\alpha_{m+1}=\beta_0=0$. The
\emph{inner corners} of $\lambda$ are defined to be the boxes whose coordinates
are
$(\alpha_1,\beta_0),(\alpha_2,\beta_1),\ldots,(\alpha_{m+1},\beta_{m})$.
Let $x_i=\beta_i-\alpha_{i+1}$
 be the contents of the inner corners for $0\leq i \leq m$.
 
{\it Example}. For $\lambda=(6,3,2,2)$ (see Figure $1$) we have  $m=3$, $(\alpha_1,\beta_1)=(4,2)$, $(\alpha_2,\beta_2)=(2,3)$ and $(\alpha_3,\beta_3)=(1,6)$. Therefore $(x_0,x_1,x_2,x_3)=(-4,0,2,6)$ and $(y_1,y_2,y_3)=(-2,1,5)$. Also note that $\sum_{0\leq i\leq 3}{x_i}-\sum_{1\leq j\leq
3}{y_j}=0$ and $\sum_{0\leq i\leq 3}{x_i}^{2}-\sum_{1\leq j\leq
3}{y_j}^{2}=26=2|\lambda|$.

\medskip
Let $\lambda$ be a partition and $x_i, y_j$ be the quantities associated to $\lambda$ defined above. 
Define (see \cite{hanxiong})
\begin{equation}\label{def:qk}
q_k(\lambda):=\sum_{0\leq i\leq m}{x_i}^{k}-\sum_{1\leq j\leq
m}{y_j}^{k}
\end{equation}  
for each integer $k$,
and  \begin{equation}
q_\nu(\lambda):=\prod_{k=1}^\ell q_{\nu_k}(\lambda)
\end{equation}
for  each partition $\nu=(\nu_1, \nu_2,
\ldots, \nu_\ell)$. \footnote{We need to be careful here: in general, $q_\nu(\lambda) = \prod_i q_{\nu_i}(\lambda) \ne \prod_i q_{\nu^i}(\lambda)$, which was used in Theorem~\ref{th:main} and commented on in footnote~\ref{foot:main}. Writing $q_\nu$ allows for a convenient access to $|\nu|$.}

We need the following results developed in \cite{hanxiong} (see (6.3), (6.6), Theorem 6.1 and Lemma 6.4 in \cite{hanxiong} respectively).

\begin{thm} \label{th:dqnu1} 
Let $\lambda^{i+}=\lambda\cup \{\square_i\}$ with $c_{\square_i}=x_i$ for $0\leq i\leq m$. 
\begin{enumerate}
\item For $k=0,1,2$, we have
\begin{equation*}
q_0({\lambda})=1,\quad q_1({\lambda})=0 \text{\quad and \quad}
q_2({\lambda})=2\,|\lambda|.
\end{equation*}

\item Suppose that $k$ is a nonnegative integer.  Then we have
$$
q_k({\lambda^{i+}})-q_k({\lambda})=\sum_{1\leq j
\leq{k/2}}2\binom{k}{2j}x_i^{k-2j}.
$$

\item Let $\nu$ be a partition. Then, there exists some $b_{\delta}\in \mathbb{Q}$ such that
\begin{equation*}
D\left(\frac{q_{\nu}(\lambda)}{H_{\lambda}}\right)= \sum_{|\delta|\leq
|\nu|-2}b_{\delta}\frac{q_\delta(\lambda)}{H_{\lambda}}
\end{equation*}
for every partition $\lambda$. 

\item Let $k$ be a nonnegative integer. Then there exists some
$b_{\delta}\in \mathbb{Q}$ such that
$$
\sum_{0\leq i\leq m}\frac{H_{\lambda}}{H_{\lambda^{i+}}}{x_i}^k=
\sum_{|\delta|\leq k}b_{\delta}q_\delta(\lambda)
$$  
for every partition
$\lambda.$
\end{enumerate}
\end{thm}

The next result is  easy to prove by Theorem \ref{th:dqnu1}(4).
\begin{thm} \label{th:h012} Let $\lambda^{i+}=\lambda\cup \{\square_i\}$ with $c_{\square_i}=x_i$ for $0\leq i\leq m$. Then,
$$
\sum_{0\leq i\leq m}\frac{H_{\lambda}}{H_{\lambda^{i+}}}{x_i}=
0
$$ 
and  
$$
\sum_{0\leq i\leq
m}\frac{H_{\lambda}}{H_{\lambda^{i+}}}{x_i}^2= |\lambda|
$$  for
every partition $\lambda.$
\end{thm}

We obtain the following result, which is a generalization of Theorem \ref{th:dqnu1}(3).

\begin{thm}  \label{th:dqnu}
 Let $\nu$  be a given partition. Then, there exists
 some $b_{\delta^0,\delta^1,\ldots,
\delta^{t-1}}\in \mathbb{Q}$ such that
\begin{equation}\label{eq:Dnu*}
D_t\left(G_{\lambda}\prod_{i=0}^{t-1}q_{\nu^i}(\lambda^i)\right)=
\sum_{(*)}b_{\delta^0,\delta^1,\ldots,
\delta^{t-1}}G_{\lambda}\prod_{i=0}^{t-1}q_{\delta^i}(\lambda^i)
\end{equation}
for every partition $\lambda$, where $(*)$ ranges over all
partitions $\delta^0,\delta^1,\ldots,
\delta^{t-1}$ such that $\sum_{i=0}^{t-1}|\delta^i|\leq
\sum_{i=0}^{t-1}|\nu^i|-2.$ Notice that
$b_{\delta^0,\delta^1,\ldots, \delta^{t-1}}$ is independent of
$\lambda$.

Furthermore, if $r\geq \frac12\sum_{i=0}^{t-1}|\nu^i|+1,$ then
$$
D_t^r\left(G_{\lambda}\prod_{i=0}^{t-1}q_{\nu^i}(\lambda^i)\right)=0
$$ 
for every partition
$\lambda$.
\end{thm}
\begin{proof}
We have
\begin{align*}
\ & \ \ \ \ D_t \left(G_{\lambda}\prod_{i=0}^{t-1}q_{\nu^i}(\lambda^i)\right)\\
&=\sum\limits_{\substack{\lambda^{+}\geq_t\lambda\\
|\lambda^{+}/\lambda|=t}}G_{\lambda^+}\prod_{i=0}^{t-1}q_{\nu^i}\bigl((\lambda^+)^i\bigr)-G_{\lambda}\prod_{i=0}^{t-1}q_{\nu^i}(\lambda^i)\\
&=G_{\lambda}\sum\limits_{\substack{\lambda^{+}\geq_t\lambda\\
|\lambda^{+}/\lambda|=t}}\frac{G_{\lambda^+}}{G_{\lambda}}\left(\prod_{i=0}^{t-1}q_{\nu^i}((\lambda^+)^i)-\prod_{i=0}^{t-1}q_{\nu^i}(\lambda^i)\right)\\
 &=\frac1tG_{\lambda}\sum_{j=0}^{t-1}  \prod\limits_{\substack{0\leq i\leq t-1\\i\neq j}}q_{\nu^i}(\lambda^i)   \sum\limits_{\substack{
|(\lambda^j)^+/\lambda^j|=1}}\frac{H_{\lambda^j}}{H_{(\lambda^j)^+}}\left(q_{\nu^j}((\lambda^j)^+)-q_{\nu^j}(\lambda^j)\right)\\
&=\frac1tG_{\lambda}\sum_{j=0}^{t-1}  \prod\limits_{\substack{0\leq
i\leq t-1\\i\neq j}}q_{\nu^i}(\lambda^i)
 H_{\lambda^j}D\left(\frac{q_{\nu^j}(\lambda^j)}{H_{\lambda^j}}\right) \\
&=
\sum_{(*)}b_{\delta^0,\delta^1,\ldots,
\delta^{t-1}}G_{\lambda}\prod_{i=0}^{t-1}q_{\delta^i}(\lambda^i)
\end{align*}
for some $b_{\delta^0,\delta^1,\ldots, \delta^{t-1}}\in\mathbb{Q}$, where $(*)$
ranges over all partitions $\delta^0,\delta^1,\ldots,
\delta^{t-1}$ such that
$\sum_{i=0}^{t-1}|\delta^i|\leq \sum_{i=0}^{t-1}|\nu^i|-2.$  The
last equality is due to Theorem \ref{th:dqnu1}(3).
\end{proof}

\begin{defi}
Let $\mu$ be a $t$-core partition.
A function $g$ of partitions  is called {\it $\mu$-admissible}, if
there exist some functions of partitions $\eta_s$ (here $\eta_s$   depends on $t, g, i, \mu$) such that 
$g(\lambda^+)-g(\lambda)=\sum_{s=0}^{s'}\eta_s(\lambda)c_{\square_i}^s$
is a polynomial of $c_{\square_i}$  for every pair of partitions
$$ 
\lambda=(\mu;
\lambda^0,\lambda^1,\ldots,\lambda^{i-1},\lambda^i,\lambda^{i+1},\ldots,\lambda^{t-1})
$$
and
$$
\lambda^+=(\mu;
\lambda^0,\lambda^1,\ldots,\lambda^{i-1},\lambda^{i}\cup\{\square_i\},\lambda^{i+1},\ldots,\lambda^{t-1}),
$$
where each coefficient $\eta_s(\lambda)$ is a linear combination of some
$\prod_{j=0}^{t-1}q_{\tau^j}(\lambda^j) $ for some partition $\tau$,\footnote{\label{foot:main}We really need here a $t$-tuple of partitions. 
We decide to reuse the quotient notation for this purpose. Different $\tau$ with the same $t$-core would give the same tuple, but the surjectivity is inconsequential here. } and for every partition $\lambda$ with $\lambda_{\tcore}=\mu$.
\end{defi}

By Theorem \ref{th:dqnu1}(2), $\lambda \mapsto \prod_{i=0}^{t-1}q_{\nu^i}(\lambda^i)$ is $\mu$-admissible for any partition $\nu$ and any $t$-core partition $\mu$.
As can be seen in Section \ref{sec:ProofMain}, the following two functions of partitions
$$
g(\lambda)=\sum\limits_{\substack{\square\in \lambda\\
h_{\square}\equiv \pm j (\modsymb t )}}h_{\square}^{2k}
$$
and
$$
g(\lambda)=\sum\limits_{\substack{\square\in \lambda\\
c_{\square}\equiv j  (\modsymb t )}}c_{\square}^{k}
$$ 
are $\mu$-admissible for 
any $t$-core partition $\mu$,  any nonnegative integer $k$, and $0\leq j\leq t-1$.
This is a consequence of  Lemmas
\ref{th:contentadd} and  \ref{th:hookdiff}.

\begin{thm} \label{th:main}
Let $t$ be a given integer, $\nu$ be a partition, and $\mu$ be a $t$-core partition.
Suppose that $g_1, g_2, \ldots, g_v$ are $\mu$-admissible functions of
partitions. 
Then, there exists some
$r\in \mathbb{N}$ such that
$$
D_t^r\left(G_\lambda \prod_{u=1}^{v}g_u(\lambda)\prod_{i=0}^{t-1}q_{\nu^i}(\lambda^i) \right)=0
$$
for every partition $\lambda$ with $\lambda_\tcore=\mu.$
Furthermore, by Theorem \ref{th:main3},
\begin{equation}
P(n)=\sum\limits_{\substack{\lambda_\tcore=\mu\\
|\lambda/\mu|=nt }} F_{\lambda/\mu} G_\lambda\, \prod_{u=1}^{v}g_u(\lambda)
\end{equation}
is a polynomial in $n$. \end{thm}

\begin{proof}[Proof of Theorem \ref{th:main}.]
Notice that this theorem is true for $v=0$ by Theorem  \ref{th:dqnu}.   We will prove it by induction. First, we have
\begin{align*}
&D_t\left(G_\lambda \prod_{u=1}^{v}g_u(\lambda)\prod_{i=0}^{t-1}q_{\nu^i}(\lambda^i)  \right)  \\ 
=\ &G_\lambda\sum\limits_{\substack{\lambda^{+}\geq_t\lambda\\
|\lambda^{+}/\lambda|=t}}\frac{G_{\lambda^+}}{G_\lambda} \left(
\prod_{u=1}^{v}g_u(\lambda^+)\prod_{i=0}^{t-1}q_{\nu^i}((\lambda^+)^i)-\prod_{u=1}^{v}g_u(\lambda)\prod_{i=0}^{t-1}q_{\nu^i}(\lambda^i)
\right)\\
=\ &G_\lambda\sum\limits_{\substack{\lambda^{+}\geq_t\lambda\\
|\lambda^{+}/\lambda|=t}}\frac{G_{\lambda^+}}{G_\lambda} \bigl(
A\cdot \Delta B + B\cdot  \Delta A + \Delta A\cdot \Delta B
\bigr),
\end{align*}
where
\begin{align*}
	A&= \prod_{u=1}^{v}g_u(\lambda),\\
	\Delta A&= \prod_{u=1}^{v}g_u(\lambda^+)-\prod_{u=1}^{v}g_u(\lambda),\\
	B&= \prod_{i=0}^{t-1}q_{\nu^i}(\lambda^i),\\
	\Delta B&= 
\prod_{i=0}^{t-1}q_{\nu^i}((\lambda^+)^i)-\prod_{i=0}^{t-1}q_{\nu^i}(\lambda^i).
\end{align*}
Next,
\begin{align*}
	&G_\lambda\sum\limits_{\substack{\lambda^{+}\geq_t\lambda\\
|\lambda^{+}/\lambda|=t}}\frac{G_{\lambda^+}}{G_\lambda} \bigl(
A\cdot \Delta B
\bigr)\\
&= \frac1t G_{\lambda} \prod_{u=1}^{v}g_u(\lambda) \sum_{j=0}^{t-1} \Biggl(\prod\limits_{\substack{0\leq i\leq t-1\\i\neq j}}q_{\nu^i}(\lambda^i)\Biggr)  H_{\lambda^j}D\Biggl(\frac{q_{\nu^j}(\lambda^j)}{H_{\lambda^j}}\Biggr), \\
&G_\lambda\sum\limits_{\substack{\lambda^{+}\geq_t\lambda\\
|\lambda^{+}/\lambda|=t}}\frac{G_{\lambda^+}}{G_\lambda} \bigl(
B\cdot \Delta A
\bigr)\\
&=\frac1t G_{\lambda} \prod_{i=0}^{t-1}q_{\nu^i}(\lambda^i) \sum_{j=0}^{t-1} \sum\limits_{\substack{ |(\lambda^j)^+/\lambda^j|=1}}\frac{H_{\lambda^j}}{H_{(\lambda^j)^+}} \\ 
&\qquad\Biggl( \prod_{u=1}^{v}g_u\bigl((\mu; \lambda^0,\ldots,(\lambda^{j})^+,\ldots,\lambda^{t-1})\bigr)-\prod_{u=1}^{v}g_u(\lambda)\Biggr), \\
&G_\lambda\sum\limits_{\substack{\lambda^{+}\geq_t\lambda\\
|\lambda^{+}/\lambda|=t}}\frac{G_{\lambda^+}}{G_\lambda} \bigl(
\Delta A \cdot \Delta B
\bigr)\\
&=\frac1tG_{\lambda}\sum_{j=0}^{t-1}     \sum\limits_{\substack{ |(\lambda^j)^+/\lambda^j|=1}}\frac{H_{\lambda^j}}{H_{(\lambda^j)^+}}\Biggl(\prod\limits_{\substack{0\leq i\leq t-1\\i\neq j}}q_{\nu^i}(\lambda^i)\Biggr)\Biggl(q_{\nu^j}((\lambda^j)^+)- q_{\nu^j}(\lambda^j)\Biggr) \\ 
& \qquad\Biggl( \prod_{u=1}^{v}g_u\bigl((\mu; \lambda^0,\ldots,(\lambda^{j})^+,\ldots,\lambda^{t-1})\bigr)-\prod_{u=1}^{v}g_u(\lambda)\Biggr).
\end{align*}

By Theorems \ref{th:dqnu1} and \ref{th:dqnu} each of the last three terms could
be written as a linear combination of some $G_\lambda
\prod_{{u'}=1}^{v'}{g'}_{u'}(\lambda)\prod_{i=0}^{t-1}{q}_{(\nu')^i}(\lambda^i)$
for some $\mu$-admissible functions ${g'}_{u'}$, where either $v'=v$ and simultaneously $\sum_{i=0}^{t-1}|(\nu')^i|\leq \sum_{i=0}^{t-1}|{\nu}^i|-2$, or $v'<v$.
Then, by Theorem \ref{th:dqnu} and induction we prove the claim.
\end{proof}

\section{Further properties of Littlewood decomposition} \label{sec:bidi} 
In this section, we always assume that $\mu$ is a given $t$-core partition
with 01-sequence $z(\mu)=(z_{\mu,j})_{j\in \mathbb{Z}}$.
For $0 \leq i\leq t-1$, we have $\mu^i=\emptyset$. Thus, $\mu^i$ has a 01-sequence equal to
$
\cdots0000\DASH11111\cdots,
$ 
which means that $z_{\mu^i,j}=0$ for $j<0$ and $z_{\mu^i,j}=1$ for $j\geq 0$.
Let 
$$
b_i:=b_i(\mu)=\min\{j\in \mathbb{Z}:j\equiv  i(\modsymb t),\
z_{\mu,j}=1\}.
$$
Then, there exist some $d_i\in \mathbb{Z}$ such that $b_i=td_i+i$ for
$0\leq i\leq t-1$.  Therefore, $z_{\mu,tj+i}=0$ for $j<d_i$ and $z_{\mu,tj+i}=1$ for $j\geq d_i$. 

\medskip

{\it Example}. The 01-sequence of the $3$-core partition $(5,3,1,1)$ is 
$$
\cdots001001\DASH1011011\cdots.
$$
Therefore,  $b_0=0$, $b_1=7$, $b_2=-4$, $d_0=0$, $d_1=2$, and $d_2=-2$. Notice that $d_0+d_1+d_2=0$.

\medskip

We define
$$B_k=\{ (i,i+k):0\leq i \leq t-1-k \}\bigcup \{ (i,i+t-k):0\leq i\leq k-1 \},$$ 
viewed as multisets for $1\leq k\leq t-1$ (\emph{i.e.},~if $k$ equals $t-k$, we keep two copies). 
Then we have
\begin{equation}\label{eqX2:Bk1}
\sum_{(i,j)\in B_k}(j-i)^2=tk(t-k).
\end{equation}

\begin{lem}\label{th:ci1}
Suppose that $\mu$ is a given $t$-core partition.
Let 
$b_i$ and $d_i$ be defined as above.
Then
$$
\sum_{i=0}^{t-1}d_i=0.
$$
\end{lem}
\begin{proof}
Let $0\leq i \leq t-1$. First, the fact that $\mu$ is a $t$-core partition implies that $z_{\mu,tj+i}=0$ for $j<d_i$ and $z_{\mu,tj+i}=1$ for $j\geq d_i$.
By the definition of 01-sequence we have
 $\#\{z_{\mu,j} =0: j \geq 0\}= \#\{z_{\mu,j} =1: j< 0\}$.
Thus, 
$$\sum_{i=0}^{t-1}d_i=\#\{z_{\mu,j} =0: j \geq 0\}- \#\{z_{\mu,j} =1: j< 0\}=0.$$
\end{proof}

For every partition $\lambda$ and  $0\leq k \leq t-1$,
recall that $\lambda(k)=\{h_{\square}\equiv k (\modsymb t):\square\in\lambda
\}$, viewed as a multiset.
\begin{lem}\label{th:cibimu}
Suppose that $\mu$ is a given $t$-core partition and $1\leq k\leq
t-1$. Then, $|\mu(0)|=0$ and
\begin{align*}
|\mu(k)|+|\mu(t-k)|=\sum\limits_{(i,j)\in B_k}\binom{d_i-d_j}{2}.
\end{align*}
\end{lem}
\begin{proof}
First, $|\mu(0)|=0$ since $\mu$ is a $t$-core partition.  By Lemma \ref{th:01sequencehookcontent}  we have
$$ |\mu(k)|=\#\{(i',j'):i'<j',j'-i'\equiv k (\modsymb t),\ z_{\mu,i'}=1,\
z_{\mu,j'}=0\}.
$$
Now we just need to consider the right hand side of the above
identity. Notice that for $0\leq i\leq t-1$ we have $z_{\mu,tj+i}=0$ for $j<d_i$ and $z_{\mu,tj+i}=1$ for $j\geq d_i$. Then, for each $(i,j)\in B_k$, we have
\begin{align*} \ &\#\{(i',j'):i'\equiv i(\modsymb t),j'\equiv j(\modsymb t),i'<j',\ z_{\mu,i'}=1,\ z_{\mu,j'}=0
\}\\ &+\#\{(i',j'):i'\equiv j(\modsymb t),j'\equiv i(\modsymb t),i'<j',\ z_{\mu,i'}=1,\ z_{\mu,j'}=0\}\\
&=1+2+\cdots+(d_i-d_j-1)=\binom{d_i-d_j}{2}=\binom{d_j-d_i+1}{2}
\end{align*} whenever $d_i> d_j$ or $d_i\leq d_j$. This means that
\begin{align*}
|\mu(k)|+|\mu(t-k)|=\sum\limits_{(i,j)\in B_k}\binom{d_i-d_j}{2}.
\end{align*}
\end{proof}

By Lemmas \ref{th:ci1} and \ref{th:cibimu} we obtain
\begin{align}
	\sum_{(i,j)\in B_k}(2j-2i)(d_i-d_j)&=t\sum_{(i,j)\in B_k}(d_i-d_j),\label{eq:Bk2}\\
	t\sum_{0\leq i\leq t-1}d_i^2&=\sum_{0\leq i < j\leq t-1}(d_i-d_j)^2,\label{eq:Bk3} \\
	-2\sum_{0\leq i\leq t-1}id_i&=\sum_{0\leq i < j\leq
t-1}(d_i-d_j),\label{eq:Bk4}
\end{align}
and
\begin{align} |\mu| &=\sum_{0\leq i< j\leq
t-1}\binom{d_i-d_j}{2}\label{eq:card:mu}\\ 
&= \frac{t}{2}\sum_{0\leq
i\leq t-1}d_i^2+\sum_{0\leq i\leq t-1}id_i \nonumber\\ 
&=\frac{1}{2t^2}\sum_{0\leq i< j\leq
t-1}\left((b_i-b_j)^2-(i-j)^2\right).\nonumber
\end{align}

\begin{lem}\label{th:lambdammum}
 Suppose that $\mu$ is a given $t$-core partition and $0\leq k\leq t-1$. Let
$\lambda$ be a partition with $\lambda_\tcore=\mu$ and
$|\lambda|=|\mu|+nt$. Then,
$$
|\lambda(k)|+|\lambda(t-k)|-|\mu(k)|-|\mu(t-k)|=2n.
$$
\end{lem}
\begin{proof}
Adding a $t$-hook to a partition $\lambda$ means that we change some $(z_{i} = 0$, $z_{i+t} = 1)$ to $(z_{i} = 1$, $z_{i+t} =0)$ in the 01-sequence of $\lambda$. Then, the lemma  follows directly from Lemma \ref{th:01sequencehookcontent} and induction on the size of $\lambda.$
\end{proof}

\section{Contents, Hook lengths and $t$-difference operators}\label{sec:ProofMain}
By the construction of 01-sequences and Lemma
\ref{th:01sequencehookcontent}, the following lemma is obvious.
\begin{lem}\label{th:01sequencecontent}
Let $z(\lambda)=(z_{\lambda,i})_{i\in \mathbb{Z}}$ be the 01-sequence of a
partition $\lambda$. Then, $\lambda$ has an inner (resp. outer)
corner with content $k$ if and only if $z_{\lambda,k-1}=0,\ z_{\lambda,k}=1$ (resp.
$z_{\lambda,k-1}=1,\ z_{\lambda,k}=0$).
\end{lem}

For each partition $\lambda$ let $\mathcal{C}(\lambda)$ be the multiset of its
contents. We explicitly compute how contents change when adding a box to one of the quotients in the Littlewood decomposition.
\begin{lem} \label{th:contentadd}
Suppose that $\mu$ is a given $t$-core partition. Let
\begin{align*}
	\lambda&=(\mu;
\lambda^0,\lambda^1,\ldots,\lambda^{i-1},\lambda^i,\lambda^{i+1},\ldots,\lambda^{t-1})\\
\noalign{{\noindent and}}
\lambda^+&=(\mu;
\lambda^0,\lambda^1,\ldots,\lambda^{i-1},\lambda^{i}\cup\{\square_i\},\lambda^{i+1},\ldots,\lambda^{t-1}).\\
\noalign{{\noindent Then,}}
\mathcal{C}(\lambda^+)\setminus \mathcal{C}(\lambda) &= \{c_{\square_i}t+b_i,
c_{\square_i}t+b_i-1,\ldots,c_{\square_i}t+b_i-(t-1) \}.
\end{align*}
\end{lem}
\begin{proof}

\goodbreak

Suppose that the $01$-sequence of $\lambda$ is $(z_{\lambda,j})_{j\in
\mathbb{Z}}$. For each $0\leq i \leq t-1$ let $(z_{\lambda^i,j})_{j\in
\mathbb{Z}}$ be the $01$-sequence of $\lambda^i$. Then, 
 $z_{\lambda^i,j}=z_{\lambda,jt+b_i}$ since $\lambda_{\tcore}=\mu$. Notice that $b_i$ are
determined by $\mu$. Let $c=c_{\square_i}$ be the content of $\square_i$. Then, by Lemma~\ref{th:01sequencecontent} adding the box
$\square_i$ to $\lambda^i$ means that
$(z_{\lambda^i,c-1},z_{\lambda^i,c})$, which is equal to $(0,1)$ in the $01$-sequence of $\lambda^i$, is
changed to $(1,0)$; or equivalently,
$(z_{\lambda,(c-1)t+b_i},z_{\lambda,ct+b_i})=(0,1)$ in the $01$-sequence of
$\lambda$ is changed to $(1,0).$ Then by
Lemma \ref{th:01sequencehookcontent} this means that we add boxes
with contents $$ct+b_i, ct+b_i-1,\ldots,ct+b_i-(t-1) $$ to
$\lambda.$
\end{proof}

By induction and Lemma \ref{th:contentadd} we obtain the following
lemma.

\begin{lem} \label{th:contentadd1}
Suppose that $\mu$ is a given $t$-core partition. Then 
$$
\mathcal{C}(\lambda)\setminus \mathcal{C}(\mu)= \bigcup_{i=0}^{t-1}
\{tc_{\square_i}+b_i-j:0\leq j\leq t-1, \square_{i} \in \lambda^i\}
$$
for every partition $\lambda$ with $\lambda_\tcore=\mu$.
\end{lem}

For the partition $\lambda$ and $0\leq i\leq t-1$, let
$x_{i,l}\ (0\leq l\leq m_i)$ be the contents of inner corners of
$\lambda^i$
and
$y_{i,l}\ (1\leq l\leq m_i)$ be the  contents of outer corners of
$\lambda^i$.

\begin{lem}\label{th:hookdiff}
Suppose that $\mu$ is a given $t$-core partition, $r$ is a given integer, $1\leq k\leq t-1$, and $0\leq i\leq t-1$.
Let
$$
\lambda=(\mu;
\lambda^0,\lambda^1,\ldots,\lambda^{i-1},\lambda^i,\lambda^{i+1},\ldots,\lambda^{t-1})
$$
and
$$
\lambda^+=(\mu;
\lambda^0,\lambda^1,\ldots,\lambda^{i-1},\lambda^{i}\cup\{\square_i\},\lambda^{i+1},\ldots,\lambda^{t-1}).
$$    
 Then
 \begin{multline*}
\sum\limits_{\substack{\square\in \lambda^+\\ h_{\square}\equiv0
(\modsymb t )}}h_{\square}^{2r} - \sum\limits_{\substack{\square\in
\lambda\\ h_{\square}\equiv0 (\modsymb t )}}h_{\square}^{2r} =\\ t^{2r}+
\sum\limits_{\substack{0\leq l \leq
    m_i}} \left(t(c_{\square_i}-x_{i,l})\right)^{2r}-
    \sum\limits_{\substack{1\leq l \leq
    m_i}} \left(t(c_{\square_i}-y_{i,l})\right)^{2r}
\end{multline*} and
\begin{align*}
\ &\sum\limits_{\substack{\square\in\lambda^+\\ h_\square\equiv
k (\modsymb t )}}h_\square^{2r}+\sum\limits_{\substack{\square\in\lambda^+\\
h_\square\equiv t-k (\modsymb t )}}h_\square^{2r} -
\sum\limits_{\substack{\square\in\lambda\\ h_\square\equiv
k (\modsymb t )}}h_\square^{2r}-\sum\limits_{\substack{\square\in\lambda\\
h_\square\equiv t-k (\modsymb t )}}h_\square^{2r}\\ &=
\sum\limits_{\substack{0\leq l \leq
    m_{i'}}} (tc_{\square_i}+b_i-tx_{i',l}-b_{i'})^{2r}-
    \sum\limits_{\substack{1\leq l \leq
    m_{i'}}} (tc_{\square_i}+b_i-ty_{i',l}-b_{i'})^{2r}
    \\ &+ \sum\limits_{\substack{0\leq
l \leq
    m_{i''}}} (tc_{\square_i}+b_i-tx_{i'',l}-b_{i''})^{2r}-
    \sum\limits_{\substack{1\leq l \leq
    m_{i''}}} (tc_{\square_i}+b_i-ty_{i'',l}-b_{i''})^{2r}
\end{align*} where $0\leq i', i''\leq
t-1$ satisfy $i'\equiv i+k (\modsymb t)$ and $i''\equiv
i-k (\modsymb t)$. Furthermore,
\begin{align*}
\sum_{\square\in \lambda^+}h_{\square}^{2r} - \sum_{\square\in
\lambda}h_{\square}^{2r} &=t^{2r} 
    +\sum_{j=0}^{t-1}\Bigl(\sum\limits_{\substack{0\leq l \leq
    m_j}} (tc_{\square_i}+b_i-tx_{j,l}-b_j)^{2r}
    \\ &-
    \sum\limits_{\substack{1\leq l \leq
    m_j}} (tc_{\square_i}+b_i-ty_{j,l}-b_j)^{2r}\Bigr).
\end{align*}
\end{lem}

\goodbreak

\begin{proof}

Suppose that the $01$-sequence of $\lambda$ is $(z_{\lambda,j})_{j\in
\mathbb{Z}}$. For each $0\leq i \leq t-1,$ let $(z_{\lambda^i,j})_{j\in
\mathbb{Z}}$ be the $01$-sequence of $\lambda^i$. Then, by the Littlewood
decomposition we have 
 $z_{\lambda^i,j}=z_{\lambda,jt+b_i}.$ Let $c=c_{\square_i}=x_{i,w_i}$.  By Lemma \ref{th:01sequencecontent}, adding the box
$\square_i$ to $\lambda^i$ means that
$(z_{\lambda^i,c-1},z_{\lambda^i,c})=(0,1)$ in the $01$-sequence of $\lambda^i$ is
changed to $(1,0)$, or equivalently,
$(z_{\lambda,(c-1)t+b_i},z_{\lambda,ct+b_i})=(0,1)$ in the $01$-sequence of
$\lambda$ is changed to $(z_{\lambda^+,(c-1)t+b_i},z_{\lambda^+,ct+b_i})=(1,0).$ 
By Lemma \ref{th:01sequencehookcontent}, to see the
difference between hook length sets of $\lambda^+$ and $\lambda,$ we
just need to consider the different pairs $(j,j')$ where $j<j',$
$(1,0)=(z_{\lambda,j},z_{\lambda,j'})$ or $(z_{\lambda^+,j},z_{\lambda^+,j'})$,  and  $\{j, j'\} \bigcap\{(c-1)t+b_i,
ct+b_i\}\neq \emptyset$ in the 01-sequences of $\lambda$ and
$\lambda^+$.

{\bf Case (i)}. First consider such pairs $(j,j')$ satisfying
$j'-j\equiv 0 (\modsymb t)$.
 Since $(z_{\lambda,(c-1)t+b_i},z_{\lambda,ct+b_i})=(0,1)$ and $(z_{\lambda^+,(c-1)t+b_i},z_{\lambda^+,ct+b_i})=(1,0)$, 
 it follows from Lemmas
\ref{th:01sequencehookcontent} and  \ref{th:01sequencecontent}  that 
the
hook lengths of $\lambda^+$ for such pairs $(j,j')$ are
\begin{align*}
    t,t(c-x_{i,l}),t(c-x_{i,l}-1),\cdots,t(c-y_{i,l+1}+1) \quad (0\leq l \leq
    w_i-1), \\  t(x_{i,l}-c),t(x_{i,l}-c-1),\cdots,t(y_{i,l}-c+1) \quad (w_i+1\leq l \leq
    m_i),
\end{align*}
and the  hook lengths of $\lambda$ for such pairs are
\begin{align*}
    t(c-x_{i,l}-1),t(c-x_{i,l}-2),\cdots,t(c-y_{i,l+1}) \quad (0\leq l \leq
    w_i-1), \\  t(x_{i,l}-c-1),t(x_{i,l}-c-2),\cdots,t(y_{i,l}-c) \quad (w_i+1\leq l \leq
    m_i).
\end{align*}
Let $|a|$ be the absolute value for every real number $a$. Thus
\begin{align*}\{&h_{\square}\in\mathcal{H}(\lambda^+):h_{\square}\equiv0
(\modsymb t )\}\setminus
\{h_{\square}\in\mathcal{H}(\lambda):h_{\square}\equiv0 (\modsymb t )\}\\
&=\left(\{t\}\bigcup\{ t(c-x_{i,l}): 0\leq l \leq
    w_i-1\}\bigcup \{ t(x_{i,l}-c): w_i+1\leq l \leq
    m_i\}\right)\\ & \setminus
    \left(\{ t(c-y_{i,l+1}): 0\leq l \leq
    w_i-1\}\bigcup \{ t(y_{i,l}-c): w_i+1\leq l \leq
    m_i\}\right) \\
&=\left(\{t\}\bigcup\{ |t(c-x_{i,l})|: 0\leq l \leq
    m_i, l\neq w_i\}\right) \setminus
    \{ |t(c-y_{i,l+1})|: 1\leq l \leq
    m_i\},\end{align*}
    which means that
\begin{multline*}
\sum\limits_{\substack{\square\in \lambda^+\\ h_{\square}\equiv0
(\modsymb t )}}h_{\square}^{2r} - \sum\limits_{\substack{\square\in
\lambda\\ h_{\square}\equiv0 (\modsymb t )}}h_{\square}^{2r} =\\ t^{2r}+
\sum\limits_{\substack{0\leq l \leq
    m_i}} \left(t(c-x_{i,l})\right)^{2r}-
    \sum\limits_{\substack{1\leq l \leq
    m_i}} \left(t(c-y_{i,l})\right)^{2r}
\end{multline*}
since $c=x_{i,w_i}$.

{\bf Case (ii)}. For $1\leq k\leq t-1$ consider such pairs
$(j,j')$ satisfying $j'-j\equiv k\ \text{or}   \ t-k  (\modsymb t)$. 
Let $0\leq i', i''\leq
t-1$ satisfy $i'\equiv i+k (\modsymb t)$ and $i''\equiv
i-k (\modsymb t)$. 
Then, by the similar argument as in Case (i) we obtain
\begin{align*}\ &\bigl( \{h_{\square}\in\mathcal{H}(\lambda^+):h_{\square}\equiv k
(\modsymb t )\}\bigcup \{h_{\square}\in\mathcal{H}(\lambda^+):h_{\square}\equiv t-k
(\modsymb t )\}\bigr)\\ 
&\setminus \bigl(
\{h_{\square}\in\mathcal{H}(\lambda):h_{\square}\equiv k (\modsymb t )\} \bigcup \{h_{\square}\in\mathcal{H}(\lambda):h_{\square}\equiv t-k (\modsymb t )\}\bigr)\\
&=\bigl(\{ |tc-tx_{i',l}+b_i-b_{i'}|: 0\leq l \leq
    m_{i'}\} \bigcup \{ |tc-tx_{i'',l}+b_i-b_{i''}|: 0\leq l \leq
    m_{i''}\}\bigr)\\
&\setminus
    \bigl(\{ |tc-ty_{i',l}+b_i-b_{i'}|: 1\leq l \leq
    m_{i'}\} \bigcup \{ |tc-ty_{i'',l}+b_i-b_{i''}|: 1\leq l \leq
    m_{i''}\}\bigr),
	\end{align*}
which means that   
\begin{align*}
\ &\sum\limits_{\substack{\square\in\lambda^+\\ h_\square\equiv
k (\modsymb t )}}h_\square^{2r}+\sum\limits_{\substack{\square\in\lambda^+\\
h_\square\equiv t-k (\modsymb t )}}h_\square^{2r} -
\sum\limits_{\substack{\square\in\lambda\\ h_\square\equiv
k (\modsymb t )}}h_\square^{2r}-\sum\limits_{\substack{\square\in\lambda\\
h_\square\equiv t-k (\modsymb t )}}h_\square^{2r}\\ &=
\sum\limits_{\substack{0\leq l \leq
    m_{i'}}} (tc+b_i-tx_{i',l}-b_{i'})^{2r}-
    \sum\limits_{\substack{1\leq l \leq
    m_{i'}}} (tc+b_i-ty_{i',l}-b_{i'})^{2r}
    \\ &+ \sum\limits_{\substack{0\leq
l \leq
    m_{i''}}} (tc+b_i-tx_{i'',l}-b_{i''})^{2r}-
    \sum\limits_{\substack{1\leq l \leq
    m_{i''}}} (tc+b_i-ty_{i'',l}-b_{i''})^{2r}.
\end{align*}


Finally, by Cases (i) and (ii) the proof is completed.\qedhere
\end{proof}

Now we are ready to prove Theorem \ref{th:main'}.
\begin{proof}[Proof of Theorem \ref{th:main'}]
By Lemmas \ref{th:contentadd} and  \ref{th:hookdiff},
the following two functions of partitions
$$
g_j(\lambda)=\sum\limits_{\substack{\square\in \lambda\\
h_{\square}\equiv \pm j (\modsymb t )}}h_{\square}^{2k}
$$
and
$$
g'_j(\lambda)=\sum\limits_{\substack{\square\in \lambda\\
c_{\square}\equiv j  (\modsymb t )}}c_{\square}^{k}
$$ 
are $\mu$-admissible for 
any $t$-core partition $\mu$,  nonnegative integer $k$, and $0\leq j\leq t-1$.
Then, the proof is achieved by Theorem \ref{th:main}.\qedhere
\end{proof}


\section{Explicit formulas for the square case}\label{sec:hanconjecture}
Now we are able to deduce more explicit results. 
\begin{thm}\label{th:hookcontentsquare}
 Suppose that $\mu$ is a given $t$-core partition. Recall that $b_i$ and $d_i$ are defined as in Section \ref{sec:bidi}. Let $1\leq k\leq t-1.$ For every partition $\lambda$ with $\lambda_\tcore=\mu$ we
 have
\begin{align*}
&\sum\limits_{\substack{\square\in\lambda\\ h_\square\equiv
k(\modsymb t )}}h_\square^2+\sum\limits_{\substack{\square\in\lambda\\
h_\square\equiv t-k(\modsymb t )}}h_\square^2 -
\Biggl(\sum\limits_{\substack{\square\in\lambda\\ c_\square\equiv
k(\modsymb t )}}c_\square^2+\sum\limits_{\substack{\square\in\lambda\\
c_\square\equiv t-k(\modsymb t )}}c_\square^2\Biggr)\\ 
=\ &\sum\limits_{\substack{\square\in\mu\\ h_\square\equiv
k(\modsymb t )}}h_\square^2+\sum\limits_{\substack{\square\in\mu\\
h_\square\equiv t-k(\modsymb t )}}h_\square^2 -
\Biggl(\sum\limits_{\substack{\square\in\mu\\ c_\square\equiv
k(\modsymb t )}}c_\square^2+\sum\limits_{\substack{\square\in\mu\\
c_\square\equiv t-k(\modsymb t )}}c_\square^2\Biggr)\\
&\quad +
 \sum_{(i,j)\in B_k}
 \Bigl(2t^2n_in_j+ t(b_{j}+j-2b_i)d_jn_i+ t(b_{i}+i-2b_j)d_in_j  \\
 &\quad -  \frac13t^2 \bigl(d_jq_3(\lambda^i)+d_iq_3(\lambda^j)\bigr)\Bigr),
\end{align*}
where $n_i=|\lambda^i|$ for $0\leq i\leq t-1.$ 
\end{thm}

\begin{proof}
We will prove this result by induction on the size of $\lambda.$
When $\lambda=\mu$, we have $\lambda^i=\emptyset$, $n_i=0$ and thus
the identity holds. Now suppose that the identity holds for
$$
\lambda=(\mu;
\lambda^0,\lambda^1,\ldots,\lambda^{i-1},\lambda^i,\lambda^{i+1},\ldots,\lambda^{t-1}).
$$
We just need to prove it for
$$
\lambda^+=(\mu;
\lambda^0,\lambda^1,\ldots,\lambda^{i-1},\lambda^{i}\cup\{\square_i\},\lambda^{i+1},\ldots,\lambda^{t-1}).
$$
Let $0\leq i', i''\leq
t-1$ satisfy $i'\equiv i+k (\modsymb t)$ and $i''\equiv
i-k (\modsymb t)$.
Set $x_{i,w_i}=c_{\square_i}.$ 
By Lemma \ref{th:hookdiff}, we have
\begin{align*}
&\sum\limits_{\substack{\square\in\lambda^+\\ h_\square\equiv k (\modsymb t )}}h_\square^2+\sum\limits_{\substack{\square\in\lambda^+\\ h_\square\equiv t-k (\modsymb t )}}h_\square^{2} - \sum\limits_{\substack{\square\in\lambda\\ h_\square\equiv k (\modsymb t )}}h_\square^2-\sum\limits_{\substack{\square\in\lambda\\ h_\square\equiv t-k (\modsymb t )}}h_\square^{2}\\
=\  &\sum\limits_{\substack{0\leq l \leq m_{i'}}} (tx_{i,w_i}+b_i-tx_{i',l}-b_{i'})^{2}- \sum\limits_{\substack{1\leq l \leq m_{i'}}} (tx_{i,w_i}+b_i-ty_{i',l}-b_{i'})^{2} \\ 
&\quad + \sum\limits_{\substack{0\leq l \leq m_{i''}}} (tx_{i,w_i}+b_i-tx_{i'',l}-b_{i''})^{2}- \sum\limits_{\substack{1\leq l \leq  m_{i''}}} (tx_{i,w_i}+b_i-ty_{i'',l}-b_{i''})^{2}\\ 
	=\  &(tx_{i,w_i}+b_i-b_{i'})^{2}-2t(tx_{i,w_i}+b_i-b_{i'})\Biggl(\sum\limits_{\substack{0\leq l \leq m_{i'}}}x_{i',l}-\sum\limits_{\substack{1\leq l \leq m_{i'}}}y_{i',l}\Biggr)\\ 
  &\quad +t^2 \Biggl(\sum\limits_{\substack{0\leq l \leq m_{i'}}}x_{i',l}^2-\sum\limits_{\substack{1\leq l \leq  m_{i'}}}y_{i',l}^2\Biggr)\\
  &\quad + (tx_{i,w_i}+b_i-b_{i''})^{2} -2t(tx_{i,w_i}+b_i-b_{i''})\Biggl(\sum\limits_{\substack{0\leq l \leq m_{i''}}}x_{i'',l}-\sum\limits_{\substack{1\leq l \leq m_{i''}}}y_{i'',l}\Biggr)\\ 
&\quad+t^2 \Biggl(\sum\limits_{\substack{0\leq l \leq m_{i''}}}x_{i'',l}^2-\sum\limits_{\substack{1\leq l \leq m_{i''}}}y_{i'',l}^2\Biggr)\\
=\ &(tx_{i,w_i}+b_i-b_{i'})^{2}+(tx_{i,w_i}+b_i-b_{i''})^{2}+2t^2n_{i'}+2t^2n_{i''}.
\end{align*}
The last equality is due to Theorem \ref{th:dqnu1}(1). 
On the other hand, by Lemma
\ref{th:contentadd} we have
\begin{align*}
\ &\sum\limits_{\substack{\square\in\lambda^+\\ c_\square\equiv
k (\modsymb t )}}c_\square^2+\sum\limits_{\substack{\square\in\lambda^+\\
c_\square\equiv t-k (\modsymb t )}}c_\square^{2} -
\Biggl(\sum\limits_{\substack{\square\in\lambda\\ c_\square\equiv
k (\modsymb t )}}c_\square^2+\sum\limits_{\substack{\square\in\lambda\\
c_\square\equiv t-k (\modsymb t )}}c_\square^{2}\Biggr)
\\ &=(tx_{i,w_i}+b_i-i')^{2}+(tx_{i,w_i}+b_i-i'')^{2}.
\end{align*}
By Theorem \ref{th:dqnu1}(3) we have
$$q_3(\lambda^{i}\cup\{\square_i\})-q_3(\lambda^i)=6x_{i,w_i}.$$

Finally, after putting those identities together we know the theorem is also true for $\lambda^+$ and thus prove the claim.
\end{proof}

We now turn to a theorem still focused on one partition at a time, but this time summing over contents or hook lengths divisible by $t$.
\begin{thm}\label{th:hookcontentsquare1}
 Suppose that $\mu$ is a given $t$-core partition. For every partition $\lambda$ with $\lambda_\tcore=\mu$ we
 have
\begin{align*}
	&\sum\limits_{\substack{\square\in\lambda\\ h_\square\equiv
0(\modsymb t )}}h_\square^2 - \sum\limits_{\substack{\square\in\lambda\\
c_\square\equiv 0(\modsymb t )}}c_\square^2\\
=\ &
 t^2\sum_{i=0}^{t-1}\left(n_i^2-d_i^2n_i-\frac13d_iq_3(\lambda^i)\right)- \sum\limits_{\substack{\square\in\mu\\
c_\square\equiv 0(\modsymb t )}}c_\square^2,
\end{align*}
where $n_i=|\lambda^i|$ for $0\leq i\leq t-1.$ 
\end{thm}

\begin{proof}
We will prove this result by induction on the size of $\lambda.$ Notice that $\mu(0)=\{h\in \mathcal{H}(\mu) \mid h\equiv 0 \pmod t\}=\emptyset$ since $\mu$ is a $t$-core partition.
When $\lambda=\mu$, we have $\lambda^i=\emptyset$, $n_i=0$ and thus
the identity holds. Now suppose that the identity holds for
$$
\lambda=(\mu;
\lambda^0,\lambda^1,\ldots,\lambda^{i-1},\lambda^i,\lambda^{i+1},\ldots,\lambda^{t-1}).
$$
We just need to prove it  for
$$
\lambda^+=(\mu;
\lambda^0,\lambda^1,\ldots,\lambda^{i-1},\lambda^{i}\cup\{\square_i\},\lambda^{i+1},\ldots,\lambda^{t-1}).
$$
Let $x_{i,w_i}=c_{\square_i}.$ By Lemma \ref{th:hookdiff}
we have
\begin{align*}
&\sum\limits_{\substack{\square\in\lambda^+\\ h_\square\equiv 0
(\modsymb t )}}h_\square^2 - \sum\limits_{\substack{\square\in\lambda\\
h_\square\equiv 0 (\modsymb t )}}h_\square^2\\
=\ &t^2+
\sum\limits_{\substack{0\leq l \leq
    m_{i}}} (tx_{i,w_i}-tx_{i,l})^{2}-
    \sum\limits_{\substack{1\leq l \leq
    m_{i}}} (tx_{i,w_i}-ty_{i,l})^{2}\\
=\ &t^2+ (tx_{i,w_i})^{2}-2t^2x_{i,w_i}\Biggl(\sum\limits_{\substack{0\leq l \leq
    m_{i}}}x_{i,l}-\sum\limits_{\substack{1\leq l \leq
    m_{i}}}y_{i,l}\Biggr)\\ 
&\quad +t^2 \Biggl(\sum\limits_{\substack{0\leq l \leq
    m_{i}}}x_{i,l}^2-\sum\limits_{\substack{1\leq l \leq
    m_{i}}}y_{i,l}^2\Biggr) \\ 
=\ &t^2+t^2
x_{i,w_i}^{2}+2t^2n_{i}.
\end{align*}
The last equality is due to Theorem \ref{th:dqnu1}(1). Also, by Lemma
\ref{th:contentadd} we have
\begin{align*}
\sum\limits_{\substack{\square\in\lambda^+\\ c_\square\equiv 0
(\modsymb t )}}c_\square^2 - \sum\limits_{\substack{\square\in\lambda\\
c_\square\equiv 0 (\modsymb t )}}c_\square^2 =t^2(x_{i,w_i}+d_i)^{2}.
\end{align*}
By Theorem \ref{th:dqnu1}(3) we have
$$q_3(\lambda^{i}\cup\{\square_i\})-q_3(\lambda^i)=6x_{i,w_i}.$$
Finally, putting these identities together, we prove the claim.
\end{proof}

It is clear that Theorem \ref{th:hookcontentsquare:mu0} is a consequence of Theorems
\ref{th:hookcontentsquare} and 
\ref{th:hookcontentsquare1}
by letting~$\mu=\emptyset$.

\medskip

Theorem \ref{th:hooksquare1} is on $t$-Plancherel averages of square sums of congruent hook lengths.

\begin{proof}[Proof of Theorem \ref{th:hooksquare1}.]
For two partitions
$$
\lambda=(\mu;
\lambda^0,\lambda^1,\ldots,\lambda^{i-1},\lambda^i,\lambda^{i+1},\ldots,\lambda^{t-1})
$$
and
$$
\lambda^+=(\mu;
\lambda^0,\lambda^1,\ldots,\lambda^{i-1},\lambda^{i}\cup\{\square_i\},\lambda^{i+1},\ldots,\lambda^{t-1}),
$$
by Case (ii) of Lemma \ref{th:hookdiff} we have
\begin{align*}
&\sum\limits_{\substack{\square\in\lambda^+\\ h_\square\equiv
k(\modsymb t )}}h_\square^2+\sum\limits_{\substack{\square\in\lambda^+\\
h_\square\equiv t-k(\modsymb t )}}h_\square^2 -
\sum\limits_{\substack{\square\in\lambda\\ h_\square\equiv
k(\modsymb t )}}h_\square^2-\sum\limits_{\substack{\square\in\lambda\\
h_\square\equiv t-k(\modsymb t )}}h_\square^2\\ 
=\ &
\sum\limits_{\substack{0\leq l \leq
    m_{i'}}} (tc_{\square_i}+b_i-tx_{i',l}-b_{i'})^{2}-
    \sum\limits_{\substack{1\leq l \leq
    m_{i'}}} (tc_{\square_i}+b_i-ty_{i',l}-b_{i'})^{2}
    \\ 
&\quad + \sum\limits_{\substack{0\leq
l \leq
    m_{i''}}} (tc_{\square_i}+b_i-tx_{i'',l}-b_{i''})^{2}-
    \sum\limits_{\substack{1\leq l \leq
    m_{i''}}} (tc_{\square_i}+b_i-ty_{i'',l}-b_{i''})^{2}
\end{align*}
where  $i'\equiv i+k (\modsymb t)$ and $i''\equiv i-k (\modsymb t).$

Let
$$
P(n):=\sum\limits_{\substack{\lambda_\tcore=\mu\\
|\lambda/\mu|=nt }}F_{\lambda/\mu}G_\lambda\Biggl(\sum\limits_{\substack{\square\in\lambda\\
h_\square\equiv
k(\modsymb t )}}h_\square^2+\sum\limits_{\substack{\square\in\lambda\\
h_\square\equiv t-k(\modsymb t )}}h_\square^2\Biggr).
$$
Then, by Lemma \ref{th:telescope} and Theorems \ref{th:dqnu1}, \ref{th:h012}  we have

\begin{align*}
&P(n+1)-P(n) \\
=\ &\frac1t\sum\limits_{\substack{\lambda_\tcore=\mu\\
|\lambda/\mu|=nt }}F_{\lambda/\mu}G_\lambda\sum_{i=0}^{t-1}\sum_{\lambda^{i+}=\lambda^{i}\cup\{\square_i\}}\frac{H_{\lambda^i}}{H_{\lambda^{i+}}}\\
&\quad \Bigl(\sum\limits_{\substack{0\leq l \leq
    m_{i'}}} (tc_{\square_i}+b_i-tx_{i',l}-b_{i'})^{2}-
    \sum\limits_{\substack{1\leq l \leq
    m_{i'}}} (tc_{\square_i}+b_i-ty_{i',l}-b_{i'})^{2}
    \\ 
&\quad + \sum\limits_{\substack{0\leq
l \leq
    m_{i''}}} (tc_{\square_i}+b_i-tx_{i'',l}-b_{i''})^{2}-
    \sum\limits_{\substack{1\leq l \leq
    m_{i''}}} (tc_{\square_i}+b_i-ty_{i'',l}-b_{i''})^{2}\Bigr)\\
 =\   & \frac1t\sum\limits_{\substack{\lambda_\tcore=\mu\\
|\lambda/\mu|=nt }}F_{\lambda/\mu}G_\lambda\sum_{i=0}^{t-1}\sum_{\lambda^{i+}=\lambda^{i}\cup\{\square_i\}}\frac{H_{\lambda^i}}{H_{\lambda^{i+}}}\\
& \quad\Bigl(
2t^2c_{\square_i}^2+2t^2|\lambda^{i'}|+2t^2|\lambda^{i''}|+
(b_i-b_{i'})^{2}+(b_i-b_{i''})^{2}\Bigr)\\
=\    &2tn+2tn+2tn + \frac1t\sum_{i=0}^{t-1}
\Bigl((b_i-b_{i'})^{2}+(b_i-b_{i''})^{2}\Bigr).
\end{align*}

On the other hand,
\begin{align*}
&\frac1t\sum_{i=0}^{t-1}
\left((b_i-b_{i'})^{2}+(b_i-b_{i''})^{2}\right)\\
=\ &\frac2t\sum_{(i,j)\in
B_k}(b_i-b_{j})^{2}\\
=\ &\frac2t\sum_{(i,j)\in B_k}\left(
t^2(d_i-d_j)^2+(i-j)^2+2t(i-j)(d_i-d_j) \right)\\
=\ &2k(t-k)+4t(|\mu(k)|+|\mu(t-k)|).
\end{align*}
The last equality is due to \eqref{eq:Bk2} and Lemma \ref{th:cibimu}.
Finally,
$$
P(n+1)-P(n)=6nt+2k(t-k)+4t|\mu(k)|+4t|\mu(t-k)|.
$$
Thus, 
\begin{align*}
	P(n)&=6t\binom{n}{2}+\bigl(2k(t-k)+4t|\mu(k)|+4t|\mu(t-k)|\bigr)n+P(0)\\
&= 6t\binom{n}{2}+\bigl(2k(t-k)+4t|\mu(k)|+4t|\mu(t-k)|\bigr)n\\
&\quad +\Biggl(\sum\limits_{\substack{\square\in\mu\\
h_\square\equiv
k(\modsymb t )}}h_\square^2 +\sum\limits_{\substack{\square\in\mu\\
h_\square\equiv t-k(\modsymb t )}}h_\square^2\Biggr). \qedhere
\end{align*}

\end{proof}

Theorem \ref{th:hooksquare2} is analogous to Theorem \ref{th:hooksquare1} when looking at $t$-Plancherel averages of square sums of hook lengths, but 
only the lengths divisible by $t$ are taken into account.

\begin{proof}[Proof of Theorem \ref{th:hooksquare2}.]
For two partitions
$$
\lambda=(\mu;
\lambda^0,\lambda^1,\ldots,\lambda^{i-1},\lambda^i,\lambda^{i+1},\ldots,\lambda^{t-1})
$$
and
$$
\lambda^+=(\mu;
\lambda^0,\lambda^1,\ldots,\lambda^{i-1},\lambda^{i}\cup\{\square_i\},\lambda^{i+1},\ldots,\lambda^{t-1}),
$$
by Case (i) of Lemma \ref{th:hookdiff} we have
\begin{align*}
	\sum\limits_{\substack{\square\in\lambda^+\\ h_\square\equiv 0 (\modsymb t)}}h_\square^2 - \sum\limits_{\substack{\square\in\lambda\\
	h_\square\equiv 0 (\modsymb t)}}h_\square^2=t^2+t^2 c_{\square_i}^{2}+2t^2n_{i}.
\end{align*}
Let
$$
P(n):=\sum\limits_{\substack{\lambda_\tcore=\mu\\
|\lambda/\mu|=nt }}F_{\lambda/\mu}G_\lambda\sum\limits_{\substack{\square\in\lambda\\
h_\square\equiv 0(\modsymb t)}}h_\square^2.
$$
Then,
\begin{align*}
P(n+1)-P(n) &=\frac1t\sum\limits_{\substack{\lambda_\tcore=\mu\\
|\lambda/\mu|=nt }}F_{\lambda/\mu}G_\lambda\sum_{i=0}^{t-1}\sum_{\lambda^{i+}=\lambda^{i}\cup\{\square_i\}}\frac{H_\lambda}{H_{\lambda^{i+}}}(t^2+t^2
c_{\square_i}^{2}+2t^2n_{i})\\
 &=t^2+3t\sum_{i=0}^{t-1}n_i\\
	 &=t^2+3nt.
\end{align*}
Thus, we have 
\begin{equation*}
P(n)=nt^2+3t\binom{n}{2}+P(0)=nt^2+3t\binom{n}{2}. \qedhere
\end{equation*}
\end{proof}

Finally, the next two theorems have similar concerns, looking at $t$-Plancherel averages, but this time for squares of contents. They are differentiated by the restriction in Theorem \ref{th:contentsquare1} or looking only at contents in one congruence class.

\begin{thm} \label{th:contentsquare1}Suppose that $\mu$ is a given
 $t$-core partition and $0\leq k\leq t-1$. For each $0\leq i\leq t-1$ let $i'$ satisfy $0\leq i'\leq t-1$ and $i-i'\equiv k(\modsymb t)$.
  Then \begin{align*} \sum\limits_{\substack{\lambda_\tcore=\mu\\
|\lambda/\mu|=nt }}F_{\lambda/\mu}G_\lambda
\sum\limits_{\substack{\square\in\lambda\\
c_\square\equiv k(\modsymb t)}}c_\square^2= t\binom{n}{2}+
    \frac1t\sum_{i=0}^{t-1}(b_i-i')^{2}n
    +\sum\limits_{\substack{\square\in\mu\\
c_\square\equiv k(\modsymb t )}}c_\square^2.\end{align*} 
\end{thm}

\begin{proof}
For two partitions
$$
\lambda=(\mu;
\lambda^0,\lambda^1,\ldots,\lambda^{i-1},\lambda^i,\lambda^{i+1},\ldots,\lambda^{t-1})
$$
and
$$
\lambda^+=(\mu;
\lambda^0,\lambda^1,\ldots,\lambda^{i-1},\lambda^{i}\cup\{\square_i\},\lambda^{i+1},\ldots,\lambda^{t-1}),
$$
by Lemma \ref{th:contentadd} we have
\begin{align*}
\sum\limits_{\substack{\square\in\lambda^+\\ c_\square\equiv k
(\modsymb t )}}c_\square^2 - \sum\limits_{\substack{\square\in\lambda\\
c_\square\equiv k (\modsymb t )}}c_\square^2 =(tc_{\square_i}+b_i-i')^{2}.
\end{align*}
Let
$$
P(n):=\sum\limits_{\substack{\lambda_\tcore=\mu\\
|\lambda/\mu|=nt
}}F_{\lambda/\mu}G_\lambda\sum\limits_{\substack{\square\in\lambda\\
c_\square\equiv k (\modsymb t )}}c_\square^2.
$$
Then, we have
\begin{align*}
P(n+1)-P(n) \
&=\frac1t\sum\limits_{\substack{\lambda_\tcore=\mu\\
|\lambda/\mu|=nt }}F_{\lambda/\mu}G_\lambda\sum_{i=0}^{t-1}\sum_{\lambda^{i+}=\lambda^{i}\cup\{\square_i\}}\frac{H_\lambda}{H_{\lambda^{i+}}}(tc_{\square_i}+b_i-i')^{2}\\
    &=tn+
    \frac1t\sum_{i=0}^{t-1}(b_i-i')^{2}.
\end{align*}
Finally, we obtain
\begin{align*}
	P(n)&=t\binom{n}{2}+
    \frac1t\sum_{i=0}^{t-1}(b_i-i')^{2}n+P(0)\\
		&=t\binom{n}{2}+
    \frac1t\sum_{i=0}^{t-1}(b_i-i')^{2}n
    +\sum\limits_{\substack{\square\in\mu\\
c_\square\equiv k(\modsymb t )}}c_\square^2.\qedhere
\end{align*}

\end{proof}

\begin{thm} \label{th:contentsquare2} Suppose that $\mu$ is a given
 $t$-core partition.
  Then we have \begin{align*} \sum\limits_{\substack{\lambda_\tcore=\mu\\
|\lambda/\mu|=nt }}F_{\lambda/\mu}G_\lambda
\sum\limits_{\substack{\square\in\lambda}}c_\square^2=
t^2\binom{n}{2}+\frac{(t^3-t)n}{6}+2tn|\mu|
+\sum\limits_{\substack{\square\in\mu}}c_\square^2.\end{align*} 
\end{thm}

\begin{proof}
 By Theorem \ref{th:contentsquare1} we have
\begin{align*}\ & \sum\limits_{\substack{\lambda_\tcore=\mu\\
|\lambda/\mu|=nt }}F_{\lambda/\mu}G_\lambda
\sum\limits_{\substack{\square\in\lambda}}c_\square^2
\\&=
t^2\binom{n}{2}+
\frac1t\sum_{i=0}^{t-1}\sum_{j=0}^{t-1}(b_i-j)^{2}n
+\sum\limits_{\substack{\square\in\mu}}c_\square^2
\\&=
t^2\binom{n}{2}+\frac{(t^3-t)n}{6}+2tn|\mu|
+\sum\limits_{\substack{\square\in\mu}}c_\square^2.\end{align*} The
last equality is due to \eqref{eq:Bk3},  \eqref{eq:Bk4} and \eqref{eq:card:mu}.
\end{proof}

\section{Acknowledgments}
The authors thank the referees for carefully reading our paper and for giving helpful comments.
P.-O. D.~and H. X.~acknowledge support from grant PP00P2\_138906 of the Swiss National Science Foundation. H. X.~also acknowledges support of a Forschungskredit FK-14-093 from the University of Z\"urich.  



\end{document}